\documentclass[12pt]{amsart}

\usepackage{amssymb}

\usepackage{enumitem}

\usepackage{graphicx}

\makeatletter
\@namedef{subjclassname@2010}{%
  \textup{2010} Mathematics Subject Classification}
\makeatother

\usepackage[T1]{fontenc}
\newtheorem{theorem}{Theorem}[section]
\newtheorem{corollary}[theorem]{Corollary}
\newtheorem{lemma}[theorem]{Lemma}

\theoremstyle{definition}
\newtheorem{definition}[theorem]{Definition}

\newtheorem{example}[theorem]{Example}

\numberwithin{equation}{section}

\newcommand{\Q}{\mathbb{Q}}
\newcommand{\Po}{\mathbb{P}}

\newcommand{\K}{\mathcal{K}}

\newcommand{\supha}{{\upharpoonright}}

\begin{document}

%%%%% To ease editing, for IMPAN journals add:

\baselineskip=17pt

%%%%%%%%%%%%%%%%

\title[Spaces with a $\mathbb{Q}$-diagonal]
{Spaces with a $\mathbb{Q}$-diagonal}

\author[Z. Feng]{Ziqin Feng}
\address{Department of Mathematics and Statistics, Auburn University, Auburn, AL 36849, USA}
\email{zzf0006@auburn.edu}

\date{}

\begin{abstract}
A space $X$ has a $\mathbb{Q}$-diagonal if $X^2\setminus \Delta$ has a $\mathcal{K}(\mathbb{Q})$-directed compact cover. We show that any compact space with a $\mathbb{Q}$-diagonal is metrizable, hence any Tychonoff space with a $\mathbb{Q}$-diagonal is cosmic. These give positive answers to Problem 4.2 and Problem 4.8 in \cite{COT11} raised by Cascales, Orihuela, and Tkachuk.
\end{abstract}

\subjclass[2010]{Primary 54E35; Secondary 03E15; 54B10}

\keywords{$\mathbb{Q}$-diagonal, metrizability, compact spaces, Tukey order, directed sets}

\maketitle

\section{Introduction}

For a directed set $P$ ordered by $\leq$, a collection $\mathcal{C}$ of subsets of a space $X$ is $P$-directed if $\mathcal{C}$ can be represented as $\{C_p: p\in P\}$ such that $C_p\subseteq C_{p'}$ whenever $p\leq p'$. Let $\mathcal{K}(M)$ be  the collection of all compact subsets of a topological space $M$. Then $\mathcal{K}(M)$ is a directed set ordered by set inclusion. A space $X$ is (strongly) dominated by $M$, or $M$-dominated,  if $X$ has a $\mathcal{K}(M)$-directed compact cover (which is cofinal in $\mathcal{K}(X)$). We say $X$ has an $M$-diagonal if $X^2\setminus \Delta$ is dominated by $M$, where $\Delta=\{(x, x): x\in X\}$. The purpose of this paper is to investigate spaces with a $\mathbb{Q}$-diagonal, where $\mathbb{Q}$ is the space of rational numbers. This is motivated by Sneider's result in \cite{S45} that any compact space $X$ with a $G_\delta$-diagonal (in the new notation, $\mathbb{R}$-diagonal where $\mathbb{R}$ is the space of real numbers) is metrizable. In \cite{M73}, it is proved that any completely regular pseudocompact space with a regular $G_\delta$-diagonal is metrizable.  In this paper, all the spaces are assumed to be completely regular and $T_1$.

The concept of $\mathbb{P}$-domination comes from the study of the geometry of topological vector spaces,  where $\Po$ is the space of irrationals. There are many applications of (strong) $\mathbb{P}$-domination in the area of functional analysis. %{\red I moved the following line to Section 2: Since $\mathbb{P}$ is homeomorphic to $\omega^\omega$, we identify $\mathbb{P}$ with $\omega^\omega$.}
One of the topological applications is to obtain metrization conditions for some class of spaces. In particular, Cascales and Orihuela in \cite{CO87} proved that any compact space $X$ with $X^2\setminus \Delta$ being strongly $\Po$-dominated is metrizable. In \cite{COT11}, it was proved that a compact space $X$ is metrizable if $X^2\setminus \Delta$ is strongly $M$-dominated for some separable metric space $M$. Recently, Gartside and Morgan in \cite{GMo16} proved that any compact space $X$ is metrizable if $X^2\setminus \Delta$ has a cofinal $P$-directed compact cover for some directed set $P$ with calibre $(\omega_1, \omega)$ (see definition in Section~\ref{prel}).
It is not known whether the word `strongly' or `cofinal' can be omitted in these results.

Cascales, Orihuela, and Tkachuk asked in \cite{COT11} whether a compact space with a $\Po$-diagonal ($\Q$-diagonal, or $M$-diagonal for some separable metric space $M$) is metrizable. In \cite{DH16}, Dow and Hart proved that compact spaces with a $\mathbb{P}$-diagonal are metrizable. Strengthening Dow and Hart's result, Guerrero S\'{a}nchez and Tkachuk in \cite{ST16} showed that any Tychonoff space with a $\mathbb{P}$-diagonal is cosmic (see definition in Section~\ref{prel}). In the same paper, they showed that under the continuum hypothesis (CH) any space with an $M$-diagonal for any separable metric space $M$ is cosmic, hence any compact space with an $M$-diagonal is metrizable. Our main positive result is that in ZFC any compact space with a $\Q$-diagonal is metrizable (Theorem~\ref{main}), hence any Tychonoff space with a $\Q$-diagonal is cosmic (Theorem~\ref{main2}). These answer Problem 4.2 and Problem 4.8 in \cite{COT11} positively. We also show that under $\mathfrak{d}>\mathfrak{b}=\omega_1$, any compact space with an $M$-diagonal for some separable metric space $M$ is metrizable (Theorem~\ref{smspace}). This gives another partial positive answer to Problem 4.3 in \cite{COT11} in addition to Guerrero S\'{a}nchez and Tkachuk's result under CH.

We will use Tukey order to compare the cofinal complexity of different $M$-diagonals. Given two directed sets $P$ and $Q$, we say $Q$ is a Tukey quotient of $P$, denoted by $P\geq_T Q$, if there is a map $\phi:P\rightarrow Q$ carrying cofinal subsets of $P$ to cofinal subsets of $Q$. In our context, where $P$ and $Q$ are both Dedekind complete (every bounded subset has a least upper bound), we have $P\geq_T Q$ if and only if there is a map $\phi: P\rightarrow Q$ which is order-preserving and such that $\phi(P)$ is cofinal in $Q$. If $P$ and $Q$ are mutually Tukey quotients, we say that $P$ and $Q$ are Tukey equivalent, denoted by $P=_T Q$. Fremlin observed that if a separable metric space $M$ is locally compact, then $\mathcal{K}(M)=_T\omega$. Its unique successor under Tukey order is the class of Polish but not locally compact spaces. For $M$ in this class, $\K(M)=_T\omega^\omega$ where $\omega^\omega$ is ordered by $f\leq g$ if $f(n)\leq g(n)$ for each $n\in \omega$.  Note that $\mathbb{P}$ is in this class, hence $\K(\mathbb{P})=_T\omega^\omega$. Gartside and Mamatelashvili in \cite{GM17} proved that $\omega^\omega\leq_T\K(\Q)\leq_T\omega^\omega\times [\omega_1]^{<\omega}$. This upper bound of $\K(\Q)$ is essential in our proof.

In Section~\ref{bair-cate}, we build a `Baire-Category' type of result in $2^{\omega_1}$ with the $G_\delta$-topology (see definition in Section~\ref{prel}). A natural question is whether we could write $2^{\omega_1}$ as a union of a $\mathfrak{c}$-sized collection of nowhere dense subsets in the $G_\delta$-topology. The answer is at least consistently `yes'. The reason is that under MA$(\omega_1)$, we have $|2^{\omega_1}|=\mathfrak{c}$. Hence under MA$(\omega_1)$, $2^{\omega_1}$ can be represented as a union of $\mathfrak{c}$-sized collection of singletons. Surprisingly, the answer is `no' if the collection of nowhere dense subsets in $G_\delta$-topology is ordered by some $\mathfrak{c}$-sized directed set. For example, Dow and Hart in \cite{DH16} showed that $2^{\omega_1}$ can't be written as a union of a $\Po$-directed collection of non-BIG compact subsets (see definition in Section~\ref{bair-cate}). The main result of Section~\ref{bair-cate} is that the same result holds with $\Po$ replaced by $\mathcal{K}(\Q)$ (Theorem~\ref{k(q)_d_b}).

Next, we will build some preliminary connections between Tukey order and $M$-domination.

\begin{lemma} Let $M$ and $N$ be two spaces such that $\mathcal{K}(M)\geq_T \mathcal{K}(N)$. Then any $N$-dominated space is $M$-dominated.
\end{lemma}

\begin{proof} Let $P=\K(M)$ and $Q=\K(N)$. Let $\phi$ be an order-preserving mapping from $P$ to $Q$ such that $\phi(P)$ is cofinal in $Q$. Suppose $X$ is $N$-dominated. Let $\{K_q:q\in Q\}$ be a compact cover of $X$ directed by $Q$. For each $p\in P$, we define $K_p=K_{\phi(p)}$. It is straightforward to see that $\{K_p:p\in P\}$ is $P$-directed. Take any $x\in X$. Then  $x\in K_q$ for some $q\in Q$. By the cofinality of $\phi(P)$, there is a $p\in P$ such that $q\leq \phi(p)$, i.e., $x\in K_{\phi(p)}$. Therefore, we obtain a $P$-directed compact cover of $X$, i.e. $X$ is $M$-dominated.   \end{proof}

We obtain the following corollary.

\begin{corollary}Let $M$ and $N$ be two spaces such that $\mathcal{K}(M)=_T \mathcal{K}(N)$. Then a space is $M$-dominated if and only if it is $N$-dominated.

\end{corollary}
%\begin{cor}Let $P$ and $Q$ be two Dedekind complete directed sets such that  $P\geq_T Q$. Then $X^2\setminus \Delta$ has a $P$-directed compact cover if it has a $Q$-directed compact cover. {\red How is this a corollary?} \end{cor}

Hence we could rephrase Dow and Hart's result in \cite{DH16} in the following way.
\begin{corollary} Let $M$ be any space such that $\K(M)\leq_T\omega^\omega$. Then a compact space $X$ is metrizable if it has an $M$-diagonal.
\end{corollary}

\begin{lemma} Let $X$ and $Y$ be two spaces such that $X$ can be continuously mapped onto $Y$. If $X$ is $M$-dominated for some space $M$, then $Y$ is $M$-dominated.
\end{lemma}

\begin{proof} Let $f$ be a continuous mapping from $X$ onto $Y$. Assume that $\{K_F:F\in \K(M)\}$ is a compact cover of $X$ ordered by $\K(M)$. Then the collection $\{f(K_F):F\in \K(M)\}$ is a compact cover of $Y$ ordered by $\K(M)$, i.e. $Y$ is $M$-dominated.
\end{proof}

%You can use this file as a template when submitting your paper to one of the IMPAN journals (except Dissertationes Mathematicae and Banach Center Publications, for which style files exist).

%The format of this file is \textbf{not} the exact final printed format (for example, the latter is scaled down, and line breaks will most often be different), but it is convenient for editing purposes.

\section{Some preliminaries}\label{prel}

Let $P$ be a directed set. A subset $C$ of $P$ is \emph{cofinal} in $P$ if for any $p\in P$, there exists a $q\in C$ such that $p\leq q$. Then $\text{cof}(P)=\min\{|C|: C\text{ is cofinal in }P\}$. We also define $\text{add}(P)=\min\{|Q|: Q \text{ is unbounded in }P\}$. Let $\kappa\geq \mu\geq \lambda$ be cardinals. We say that $P$ has \emph{calibre~$(\kappa, \mu, \lambda)$} if for every $\kappa$-sized subset $S$ of $P$ there is a $\mu$-sized subset $S_0$ such that every $\lambda$-sized subset of $S_0$ has an upper bound in $P$. We write calibre~$(\kappa, \mu, \mu)$ as calibre~$(\kappa, \mu)$ and calibre~$(\kappa, \kappa, \kappa)$ as calibre~$\kappa$. If $\mathcal{C}=\{K_p: p\in P\}$ is a $P$-directed collection of subsets of $X$ and $Q$ is a subset of $P$, we denote $K[Q]=\bigcup\{K_q: q\in Q\}$.

Since $\mathbb{P}$ is homeomorphic to $\omega^\omega$, we identify $\mathbb{P}$ with $\omega^\omega$. For any $f, g\in \omega^\omega$, we say that $f\leq^\ast g$ if the set $\{n\in\omega:f(n)>g(n) \}$ is finite. Then $\mathfrak{b}=\text{add}(\omega^\omega, \leq^\ast)$ and $\mathfrak{d}=\text{cof}(\omega^\omega, \leq^\ast)$. See \cite{D84} for more information about small cardinals. For any $f\in \omega^\omega$ and $n\in \omega$, let $f{\upharpoonright}_n$ denote the restriction of $f$ on $n$. Also, let $[f\supha_n]=\{g\in \omega^\omega: f\supha_n \subset g\}$.

For any subset $I$ of $\omega_1$, let $\text{Fn}(I, 2)$ be the set of finite partial functions from $I$ to $2$, and $2^{<\omega_1}$ the binary tree of countable sequences of zeros and ones. For any $s\in \text{Fn}(\omega_1, 2)$, $[s]=\{x\in 2^{\omega_1}: s\subset x\}$; similarly, for any node $\rho \in 2^{<\omega_1}$, $[\rho]=\{x\in 2^{\omega_1}: \rho\subset x\}$. Then $\{[s]: s\in \text{Fn}(\omega_1, 2)\}$ is the standard base for the product topology on $2^{\omega_1}$ and $\{[\rho]: \rho\in 2^{<\omega_1}\}$ is the standard base for the $G_\delta$-topology on $2^{\omega_1}$. For any two elements $\rho$, $\rho'\in 2^{<\omega_1}$, we say $\rho'$ is an extension of $\rho$ if $\rho\subseteq \rho'$. We say that $\rho'$ is a finite (countable) extension of $\rho$ if $\rho'\setminus \rho$ is finite (respectively, countable).

We say that a space $X$ is $\omega$-bounded if the closure of any countable subset of $X$ is compact; $X$ is cosmic if it has a countable network.

\section{Properties (T) and (wT)}\label{t_wt}

 Dow and Hart in \cite{DH16} introduced a very useful property of subsets of $2^{\omega_1}$.  A subset $Y$ of $2^{\omega_1}$ is said to be \emph{BIG}  if it is compact and projects onto some final product, i.e., there is a $\delta\in \omega_1$ such that $\pi_{\delta}[Y]=2^{\omega_1\setminus \delta}$; here $\pi_\delta[Y]=\{y\supha_{\omega_1\setminus \delta}: y\in Y  \}$. This condition could also expressed as follows: there is a $\delta\in \omega_1$ such that for every $s\in \text{Fn}(\omega_1\setminus \delta, 2)$, the intersection $Y\cap [s]$ is nonempty. We say a subset of $2^{\omega_1}$ is non-BIG if it is not BIG.

It is straightforward to see that if $Y$ is a compact subset of $2^{\omega_1}$ and $[\rho]\subset Y$ for some node $\rho\in 2^{<\omega_1}$, then $Y$ is BIG witnessed by the supremum of $\text{Dom}(\rho)$. The following lemma is proved by Dow and Hart in \cite{DH16}.

\begin{lemma}\label{b-base} If $Y$ is BIG, then there is a node $\rho$ in the tree $2^{<\omega_1}$ such that $[\rho]\subseteq Y$.
\end{lemma}

 Thus a BIG subset of $2^{\omega_1}$ has a non-empty interior in the $G_\delta$-topology. Hence if a subset of $2^{\omega_1}$ is nowhere dense in the $G_\delta$-topology, then it is non-BIG.

Under some conditions, the property of being BIG is preserved by taking countable intersections. The following lemma will be useful.

\begin{lemma}\label{ctbl_insctn_big} Suppose $\alpha\in\omega_1$. Let $\mathcal{B}=\{B_\beta: \beta<\alpha\}$ be a countable decreasing collection of BIG subsets of $2^{\omega_1}$ where the BIGness of $B_\beta$ is witnessed by $\delta_\beta$ for each $\beta<\alpha$.

Then $\bigcap \mathcal{B}$ is BIG witnessed by $\gamma=\sup\{\delta_\beta: \beta<\alpha\}$.
\end{lemma}
\begin{proof} It is clear that $\bigcap \mathcal{B}$ is compact. If $\alpha$ is a successor ordinal, the result trivially holds.

Now assume that $\alpha$ is a limit ordinal in $\omega_1$. We show that for any $s\in\text{Fn}(\omega_1\setminus \gamma, 2)$, $[s]\cap \bigcap\mathcal{B}\neq \emptyset$. Take an arbitrary $s\in\text{Fn}(\omega_1\setminus \gamma, 2)$. Since $\gamma=\sup\{\delta_\beta: \beta<\alpha\}$,  $[s]\cap B_\beta\neq\emptyset$ for each $\beta<\alpha$. Then $\{[s]\cap B_\beta: \beta<\alpha\}$ is a decreasing sequence of compact subsets of $2^{\omega_1}$. Therefore, $\bigcap\{[s]\cap B_\beta: \beta<\alpha\}$ is non-empty, i.e. $[s]\cap \bigcap\{B_\beta: \beta<\alpha \}$ is non-empty. Hence $\bigcap \mathcal{B}$ is BIG witnessed by $\gamma$.  \end{proof}

%Gartside and Mamatelashvili showed that $\omega^\omega \times \omega_1\leq_T\mathcal{K}(\mathbb{Q})\leq_T\omega^\omega \times \omega_1^{<\omega}$.

We start by introducing some new classes of `Tiny' subsets and `Fat' subsets of  $2^{\omega_1}$.

\begin{definition} Let $A$ be a subset of $2^{\omega_1}$. We say that $A$ has:
\begin{itemize}

\item[]\emph{Property (T)} if for any node $\rho\in 2^{<\omega_1}$, there is a \emph{finite} extension $\hat{\rho}\in 2^{<\omega_1}$ of $\rho$ such that $[\hat{\rho}]\cap A=\emptyset$.

\item[]\emph{Property (wT)} if for any node $\rho\in 2^{<\omega_1}$, there is an extension $\hat{\rho}\in 2^{<\omega_1}$ of $\rho$ such that $[\hat{\rho}]\cap A=\emptyset$.

\item[]\emph{Property (F)} if there exists a node $\rho\in 2^{<\omega_1}$ such that $[\rho]\subseteq  A$.

\end{itemize}
\end{definition}

By the definition, we see that any subset of $2^{\omega_1}$ with property (T) has property (wT). We say that a subset of $2^{\omega_1}$ has property \emph{non-(T)} (respectively, \emph{non-(wT)}) if it doesn't have property (T)(respectively, (wT)).  We see that  a subset $B$ of $2^{\omega_1}$ has property non-(T) (property non-(wT)) if there exists a node $\rho\in 2^{<\omega_1}$ such that $[\hat{\rho}]\cap B \neq \emptyset$ for any finite (respectively, countable) extension $\hat{\rho}$ of $\rho$. Also, it is straightforward to see that properties (T) and (wT) are hereditary.

Next, we investigate the relation between non-BIGness and property (T). Also, we prove that for any subset of $2^{\omega _1}$ which is nowhere dense in the $G_\delta$-topology has property (wT).

\begin{lemma}\label{nb-t} Let $A$ be a compact subset of $2^{\omega_1}$. Then $A$ is non-BIG if and only if it has property (T).\end{lemma}

\begin{proof} First, assume that $A$ is non-BIG. Take a node $\rho\in 2^{<\omega_1}$. Since $[\rho]$ is compact, so is $[\rho]\cap A$. Hence it is closed in $[\rho]$. Since $A$ is non-BIG, $[\rho]\setminus A\neq \emptyset$. Pick $x\in [\rho]\setminus A$. Then since $[\rho]\setminus A$ is open in $[\rho]$, there exists a finite extension $\hat{\rho}$ of $\rho$ such that $x\in [\hat{\rho}]$ and $[\hat{\rho}]\cap A=\emptyset$. Hence $A$ has property (T).

Now, assume that $A$ is BIG. By Lemma~\ref{b-base}, there is a node $\rho_0$ in $2^{<\omega_1}$ such that $[\rho_0]\subseteq A$. Hence any finite extension $\hat{\rho_0}$ of $\rho_0$ satisfies that $[\hat{\rho_0}]\subseteq A$. Therefore, $A$ has property non-(T).   \end{proof}

\begin{lemma}\label{wt-nwd} Let $A$ be a subset of $2^{\omega_1}$. Then $A$ has property (wT) if  it is nowhere dense in the $G_\delta$-topology.\end{lemma}

\begin{proof}  Let $C$ be the closure of $A$ in the $G_\delta$-topology.  Take a node $\rho\in 2^{<\omega_1}$. Then $[\rho]\setminus C\neq \emptyset$. Pick $x\in [\rho]$ such that $x\notin C$. Since $C$ is also closed in $[\rho]$ with $G_\delta$-topology, there exists a countable extension $\hat{\rho}$ of $\rho$ such that $\hat{\rho}\subset x$ and  $[\hat{\rho}]\cap C =\emptyset$. Then $C$ has property (wT), hence so does $A$ . \end{proof}

Next, we show that property (wT) is preserved by taking countable union and under some nice condition, the countable intersection of subsets with property non-(wT) has property non-(T).
\begin{lemma}\label{ctbl_u_wt_wt} Let $\{A_n: n\in \omega\}$ be a countable collection of subsets in $2^{\omega_1}$ with property (wT). Then $\bigcup \{A_n: n\in \omega\}$ has property (wT).
\end{lemma}

\begin{proof} Take a node $\rho\in 2^{<\omega_1}$. Inductively, we can build inductively a sequence $\{\gamma_n: n\in \omega\}$ in $2^{<\omega_1}$ such that: i) $\gamma_0$ extends $\rho$ with $[\gamma_0]\cap A_0=\emptyset$; ii) $\gamma_n$ extends $\gamma_{n-1}$ with $[\gamma_n]\cap A_n=\emptyset$. Let $\hat{\rho}=\bigcup\{\gamma_n: n\in \omega\}$. It is clear that $[\hat{\rho}]\cap (\bigcup \{A_n: n\in \omega\})=\emptyset $, hence $\bigcup \{A_n: n\in \omega\}$ has property (wT). \end{proof}

\begin{lemma}\label{ctbl_insctn_nwt_nt} Let $\{B_n: n\in \omega\}$ be a decreasing collection of subsets in $2^{\omega_1}$ with property non-(wT). If every sequence $\{x_n: x_n\in B_n\}$ has a cluster point in  $\bigcap\{B_n: n\in \omega\}$, then $\bigcap \{B_n: n\in \omega\}$ has property non-(T). % (The proof is NOT right. Still think the result is correct. At least I hope it is correct! Zorn's Lemma)
\end{lemma}

\begin{proof}% We call a countable sequence  $S$ (could be finite) in $2^{\omega_1}$ is GOOD if $S(i)\subset S(i+1)$ for all the $i$'s with $S(i+1)$ defined and $S(i)$ witness the non-(wT) property of $B_i$. Naturally, we say the length of such a sequence is $n$ if $|S|=n$; $\infty$ otherwise. Let $\mathcal{Z}=\{Z: S \text{ is GOOD for any }S\in Z\}$. Order $\mathcal{Z}$ by set-inclusion. For any chain $\mathcal{C}$ in $\mathcal{Z}$, we see that $\bigcup\mathcal{C}\in \mathcal{Z}$ and $\bigcup\mathcal{C}$ is also an upper bound of $\mathcal{C}$. By Zorn's Lemma, there is a maximal element $Z^\ast$ in $\mathcal{Z}$.

Let $G=\bigcap \{B_n: n\in \omega\}$. First, we claim that $\overline{G}=\bigcap \{\overline{B_n}: n\in \omega\}$. Since $G\subset B_n\subset \overline{B_n}$ for each $n\in \omega$, $G\subset \bigcap \{\overline{B_n}: n\in \omega\}$. Hence $\overline{G}\subseteq \bigcap \{\overline{B_n}: n\in \omega\}$. Next, we prove that $\bigcap \{\overline{B_n}: n\in \omega\}\subseteq\overline{G}$. Take an $x\in \bigcap \{\overline{B_n}: n\in \omega\}\setminus G$. Fix a finite subset $s$ of $x$. We see that $[s]$ is an open neighborhood of $x$. So for each $n$, there is an $x_n\in B_n\cap [s]$. By assumption, there is a cluster point $y$ of $\{x_n: n\in \omega\}$ which lies in $G$. It is straightforward to see that $y$ also lies in $[s]$. Hence $[s]\cap G\neq \emptyset$. Therefore $x\in \overline{G}$.  Hence $\bigcap\{\overline{B_n}: n\in \omega\}\subset \overline{G}$.

Next we show that $G$ has property non-(T). For each $n$, $\overline{B_n}$ has property non-(wT) since $B_n$ does. So $\overline{B_n}$ has property non-(T), hence it is BIG by Lemma~\ref{nb-t}. By Lemma~\ref{ctbl_insctn_big},  $\bigcap \{\overline{B_n}: n\in \omega\}$ is BIG, hence so is $\overline{G}$. Then by Lemma~\ref{b-base}, there is a node $\rho\in 2^{<\omega_1}$ such that $[\rho]\subset \overline{G}$. Take an arbitrary $s\in\text{Fn}(\omega_1\setminus \text{Dom}(\rho), 2)$. We claim that $[\rho\cup s]\cap G\neq \emptyset$. Fix an $x\in [\rho\cup s]$. Note that $[\rho\cup s]\subseteq [\rho]\subseteq \overline{G}$. Clearly, $x\in \overline{G}$. Fix an increasing sequence $\{\rho_n: n\in\omega\}$ such that $\rho_n$ is a finite subset of $\rho$ for each $n\in\omega$ and $\rho=\bigcup\{\rho_n:n\in \omega\}$. Then for each $n\in\omega$, $([s\cup \rho_n])\cap \overline{G}\neq \emptyset$ because $[s\cup\rho_n]$ is an open neighborhood of $x$. We pick $y_n\in G\cap ([s\cup \rho_n])$. Notice that $y_n\in B_n$ for each $n$. There is a cluster point $y$ of $\{y_n: n\in \omega\}$ in $G$. It is straightforward to verify that $y\in [\rho]\cap [s]$. Hence $[\rho\cup s]\cap G\neq \emptyset$.  Therefore $G$ has property non-(T) witnessed by $\rho$. \end{proof}

\section{$\mathcal{K}(\mathbb{Q})$-directed compact cover of $2^{\omega_1}$}\label{bair-cate}

In this section, we prove that no $\K(\Q)$-directed collection of subsets of $2^{\omega_1}$ can cover the whole space if the elements in the collection are compact and non-BIG. We will divide the proof into three cases: 1) $\mathfrak{d}=\aleph_1$; 2) $\mathfrak{b}>\aleph_1$; 3) $\mathfrak{d}>\mathfrak{b}=\aleph_1$. Under the assumption $\mathfrak{d}>\mathfrak{b}=\aleph_1$, the result holds with $\Q$ replaced by any separable metric space.

\begin{lemma}\label{d=w1} Let $P$ be a directed set with $\text{cof}(P)=\aleph_1$. If $2^{\omega_1}$ has a $P$-directed compact cover $\mathcal{C}=\{K_p:p\in P\}$, then there exists a $p\in P$ such that $K_p$ is BIG.   \end{lemma}

\begin{proof} Suppose, for contradiction, that $K_p$ is non-BIG for each $p\in P$. By Lemma~\ref{nb-t}, $K_p$ has property (T) for each $p\in P$. Since $\text{cof}(P)=\aleph_1$, there exists a cofinal subset $\{p_\alpha: \alpha\in \omega_1\}$ of the directed set $P$. Clearly $\{K_{p_\alpha}: \alpha\in\omega_1\}$ covers $2^{\omega_1}$. Since $K_{p_\alpha}$ has property (T) for each $\alpha\in\omega_1$, we can inductively define a sequence $\rho_\alpha\in 2^{<\omega_1}$ such that i) $\rho_\alpha$ extends $\rho_\beta$ for all $\beta<\alpha$; ii) $[\rho_\alpha]\cap K_{p_\alpha}=\emptyset$. Then $\rho=\bigcup\{\rho_\alpha: \alpha<\omega_1\}$ is well-defined. Then any $x\in 2^{\omega_1}$ which is an extension of $\rho$ is not in $\bigcup\{K_{p_\alpha}: \alpha<\omega_1\}$. This contradiction finishes the proof. \end{proof}

%A subset $D$ of $2^{\omega_1}$ is FAT, if there exists a $\rho\in 2^{<\omega_1}$ such that $[\rho]\subseteq D$, equivalently, $D$ has a nonempty interior in the $G_\delta$-topology of $2^{\omega_1}$.

\begin{lemma}\label{Pcover-wt} Let $\{K_f: f\in \omega^\omega\}$ be an $\omega^\omega$-directed collection of compact subsets of $2^{\omega_1}$. If  $\bigcap\{K[[f\supha_n]]: n\in \omega\}$ has property (T) for each $f\in \omega^\omega$, then  $\bigcup\{K_f:f\in \omega^\omega\}$ has property (wT). \end{lemma}  %{\red Have you defined $K[\sigma]$ for $\sigma\in \omega^{<\omega}$?}

\begin{proof}  %It is easy to see that $K_f$ has property (T) for each $f\in \omega^\omega$.
First, we fix an $f\in \omega^\omega$. We claim that there exists an $n\in \omega$ such that $K[[f\supha_n]]$ has property (wT). Suppose not. We show that $\bigcap\{K[[f\supha_n]]: n\in \omega\}$  has property non-(T), which contradicts the assumption. By Lemma~\ref{ctbl_insctn_nwt_nt}, it is sufficient to show that any sequence $\{x_n: x_n\in K[[f\supha_n]]\}$ has a cluster point in $\bigcap\{K[[f\supha_n]]: n\in \omega\}$.  Let $\{x_n: n\in\omega\}$ be such a sequence. Then for each $n\in\omega$, there exists a $g_n\in [f\supha_n] $ such that $x_n\in K_{g_n}$. For each $m\in\omega$, we define a function $h_m\in\omega^\omega$ in the following way: for each $i\in \omega$, $h_m(i)=\max\{g_n(i): n\geq m\}$. It is straightforward to verify that $\{h_m:m\in\omega\}$ is a well-defined decreasing sequence and $h_m\geq g_n$ for all $n\geq m$. Hence we obtain a sequence of compact subsets $K_{h_m}$ such that $h_m\in [f\supha_m]$ for each $m$,  $x_i\in K_{h_m}$ for $i\geq m$, and $h_m\geq f$ for each $m$. By the compactness of $K_{h_0}$, $\{x_m: m\in \omega\}$ has a cluster point, namely, $x$. Since $K_{h_m}$ is compact for each $m$, we see that $x\in K_{h_m}$ for each $m$, hence, $x\in \bigcap \{K_{h_m}: m\in \omega\}$. Since $K_{h_m}\subset K[[f\supha_m]]$, $x\in \bigcap\{K[[f\supha_m]]: m\in\omega\}$. Hence by Lemma~\ref{ctbl_insctn_nwt_nt}, $\bigcap\{K[[f\supha_n]]: n\in \omega\}$ has property non-(T).

For each $f\in \omega^\omega$, let $n_f$ be the number such that $K[[f\supha_{n_f}]]$ has property (wT). By the Lindel\"{o}f property of $\omega^{\omega}$, we can get a sequence $f_i$ such that $\{[f_i\supha_{n_{f_i}}]: i\in \omega\}$ is an open cover of $\omega^\omega$. Hence $\bigcup\{K_f:f\in \omega^\omega\}=\bigcup \{K[[f_i\supha_{n_{f_i}}]]: i\in \omega\}$. By Lemma~\ref{ctbl_u_wt_wt}, $\bigcup \{ K[[f_i\supha_{n_{f_i}}]]: i\in \omega\}$ has property (wT). Therefore $\bigcup\{K_f:f\in \omega^\omega\}$ has property (wT).  \end{proof}

%\begin{lemma}Let $\{K_f: f\in \omega^\omega\}$ be an $\omega^\omega$-ordered collection of compact subsets of $2^{\omega_1}$. Suppose that for any decreasing sequence $g_n\in\omega^\omega$ which converges to $g$, we have that $\bigcap \{K_{g_n}: n\in \omega\}=K_g$, then $K_g=\bigcap\{K[f\upharpoonright_n]: n\in \omega\}$. \end{lemma}

%\begin{proof} Take any $x\in \bigcap\{K[f\upharpoonright_n]: n\in \omega\}$. For each $n\in \omega$, pick $g_n\in [f\upharpoonright_n]$ such that  $x\in K_{g_n}$. Then for each $m$, we define $h_m(n)=\max\{g_{j}(n): j\geq m\}$. This is well defined because $g_j(n)=f(n)$ for all $j>n$. Also, we see that $h_m$ is a decreasing sequence in $\omega^\omega$ which converges to $f$. Hence $x\in \bigcap\{K_{g_n}: n\in \omega\}\subset\bigcap\{K_{h_m}: m\in \omega\}=K_f$.\end{proof}

\begin{lemma}\label{int_base_t}Assume $\mathfrak{b}>\aleph_1$. Let $\{K_f: f\in \omega^\omega\}$ be an $\omega^\omega$-directed collection of compact subsets of $2^{\omega_1}$.

 Suppose that $K_f$  has property (T) for each $f\in \omega^\omega$.  Then for each $f\in \omega^\omega$,  $\bigcap\{K[[f\supha_n]]: n\in \omega\}$ has property (T). \end{lemma}

\begin{proof} We fix $f\in \omega^\omega$. Let $G = \bigcap\{K[[f\supha_n]]: n\in \omega\}$ for convenience.

Suppose, for a contradiction, that $G$ has property non-(T) witnessed by the node $\rho\in 2^{<\omega_1}$, i.e., for any $s\in \text{Fn}(\omega_1\setminus \text{Dom}(\rho), 2)$, $([\rho]\cap [s])\cap G\neq \emptyset$.

Fix an $s\in \text{Fn}(\omega_1\setminus \text{Dom}(\rho), 2)$. For each $n$, since $([\rho\cup s])\cap K[[f\supha_n]]\neq \emptyset$, pick $h_{s, n}\in [f\supha_n]$ such that $([\rho\cup s])\cap K_{h_{s, n}}\neq \emptyset$. Then we define $h_s(m)=\max\{h_{s, n}(m): n\in\omega \}$ for each $m\in\omega$. The function $h_s$ is well-defined since $h_{s, n}(m)=f(m)$ for all $n\geq m$. We see that $h_s\geq h_{s, n}$ for each $n\in \omega$.

The collection $\{h_s: s\in \text{Fn}(\omega_1\setminus \text{Dom}(\rho), 2) \}$ has cardinality $\aleph_1$. As $\mathfrak{b}>\aleph_1$, there exists an $h\in \omega^\omega$ such that $h\geq f$ and $h\geq^\ast h_s$ for all $s\in \text{Fn}(\omega_1\setminus \text{Dom}(\rho), 2)$.

  We will obtain a contradiction by showing that $K_h$ has property non-(T) witnessed by $\rho$.  To this end, let $s\in \text{Fn}(\omega_1\setminus \text{Dom}(\rho), 2)$.  We claim that $([\rho\cup s])\cap K_h\neq \emptyset$.
Since $h\geq^* h_s$, choose $m\in \omega$ such that $h(i)\geq h_s(i)$ for all $i\geq m$. Then for each $i<m$, $h(i)\geq f(i)=h_{s, m}(i)$; and for each $i\geq m$, $h(i)\geq h_s(i)\geq h_{s, m}(i)$. Hence $h\geq h_{m, s}$, i.e. $K_h\supseteq K_{h_{s, m}}$. Therefore, $([\rho\cup s])\cap K_h\neq \emptyset$.  \end{proof}

By Lemma~\ref{Pcover-wt} and Lemma~\ref{int_base_t}, we obtain the following theorem.

\begin{theorem}\label{b>w1} Assume $\mathfrak{b}>\aleph_1$. Let $\{K_f: f\in \omega^\omega\}$ be an $\omega^\omega$-directed collection of compact subsets of $2^{\omega_1}$ with property (T).

Then $\bigcup\{K_f: f\in \omega^\omega\}$ has property (wT).
\end{theorem}

We also obtain some results in ZFC. Let $P$ be a directed set and $\{K_p: p\in P\}$ a $P$-directed compact cover of $2^{\omega_1}$. A set-valued function $\psi$ from $P$ to $2^{\omega_1}$ is induced by $\{K_p: p\in P\}$ in the following way:  $\psi(p)=K_p$ for any $p\in P$.   %The compact cover $\{K_p: p\in P\}$ is a set-valued mapping from $P$ to $2^{\omega_1}$.

\begin{lemma} Let $P$ be a directed set equipped with a first countable topology such that any convergent sequence is bounded.

Let $\{K_p: p\in P\}$ be a $P$-directed compact cover of $2^{\omega_1}$
such that $K_p$ has property (T) for each $p\in P$. Suppose that the mapping induced by $\{K_p: p\in P\}$ is upper semi-continuous. Then for each $p\in P$, $K_p=\bigcap\{K[B_n]: n\in \omega\}$ where $\{B_n: n\in\omega\}$ is a countable local base at $p$ and $K[B_n]=\bigcup\{K_q:q\in B_n\}$.
\end{lemma}

\begin{proof} Fix a $p\in P$. It is clear that $K_p\subseteq \bigcap\{K[B_n]: n\in \omega\}$. Suppose that $\bigcap\{K[B_n]: n\in \omega\}\setminus K_p\neq\emptyset$. Fix $x\in \bigcap\{K[B_n]: n\in \omega\}\setminus K_p$. For each $n\in \omega$, we take $q_n\in B_n$ such that $x\in K_{q_n}$. Pick an open set $U$ such that $x\notin U$ and $K_p\subset U$. By the upper semi-continuity of the induced set-valued mapping, there is an $n_0$ such that $K[B_{n_{0}}]\subset U$. Then we get $x\in K_{q_{n_0}}\subset K[B_{n_0}]\subset U$ which is a contradiction.
\end{proof}

\begin{lemma} Let $\{K_f: f\in \omega^\omega\}$ be an $\omega^\omega$-directed collection of compact subsets of $2^{\omega_1}$.
 Suppose that $K_f$ has property (T) for each $f\in \omega^\omega$.  Then for each $f\in \omega^\omega$,  $\bigcap\{K[[f\supha_n]]: n\in \omega\}$ has property non-(F).

\end{lemma}

\begin{proof}Fix an $f\in \omega^\omega$. For convenience, let $G=\bigcap\{K[[f\supha_n]]: n\in \omega\}$.  Suppose, for a contradiction, that there is a $\rho\in 2^{<\omega_1}$ such that $[\rho]\subseteq G$. Since $[\rho]$ is separable, we pick a countable dense subset $D=\{d_n: n\in\omega\}$, i.e.,  for any $s\in \text{Fn}(\omega_1\setminus \text{Dom}(\rho), 2)$, $[\rho\cup s] \cap D\neq \emptyset$.  For each $n$, we can pick $g_n\in [f\supha_n]$ such that $d_n\in K_{g_n}$ and $g_n\geq f$. Define $h$ such that $h(m)=\max\{g_n(m): n\in \omega\}$ for each $m\in\omega$. It is clear that $h\geq g_n$ for each $n$, i.e. $K_h\supseteq K_{g_n}$ for each $n$. Hence $D\subset K_h$. Therefore $K_h$ has property non-(T) witnessed by $\rho$ which is a contradiction.
\end{proof}

Now we are ready to prove a general result under the assumption that $\mathfrak{d}=\aleph_1$ or $\mathfrak{b}>\aleph_1$. %$2^{\omega_1}$ can not resprent as a uinion of $\omega^\omega\times P$-directed collection of compact non-BIG subsets if cof$(P)=\aleph_1$.

\begin{theorem}\label{d=w1orb>w1} Assume either $\mathfrak{d}=\aleph_1$ or $\mathfrak{b}>\aleph_1$.  If $2^{\omega_1}$ has a $\omega^\omega\times P$-directed compact cover $\{K_{(f, p)}: f\in \omega^\omega, p\in P\}$ with cof($P$)$=\aleph_1$, then there exist $f\in \omega^\omega$ and $p\in P$ such that $K_{(f, p)}$ is BIG.
\end{theorem}

\begin{proof}
If $\mathfrak{d}=\aleph_1$, then $\text{cof}(\omega^\omega\times P)=\aleph_1$. Then the result holds by Lemma~\ref{d=w1}.

Now assume that $\mathfrak{b}>\aleph_1$. Suppose that $K_{(f, p)}$ is non-BIG for any $f\in \omega^\omega$ and $p\in P$. By Lemma~\ref{nb-t}, $K_{(f, p)}$ has property (T) for any $f\in \omega^\omega$ and $p\in P$.  By Theorem~\ref{b>w1}, for each $p\in P$, $\bigcup\{K_{(f, p)}: f\in \omega^\omega\}$ has property (wT). Fix a dominating subset $\{p_\alpha: \alpha\in\omega_1\}$ of $P$. Let $K[p_\alpha]=\bigcup \{K_{(f, p_\alpha)}: f\in \omega^\omega\}$ for each $\alpha\in\omega_1$. Then $K[p_\alpha]$ has property (wT) for each $\alpha\in\omega_1$.  Let $\rho_0\in 2^{<\omega_1}$ with $[\rho_0]\cap K[p_0]=\emptyset$. Inductively build a sequence of $\{\rho_\alpha: \alpha\in\omega_1\}$ in $2^{<\omega_1}$ such that: i) $[\rho_\alpha]\cap K[p_\alpha]=\emptyset $ for each $\alpha\in\omega_1$; ii) $\rho_\alpha\supseteq \rho_\beta$ for each $\beta\leq \alpha$. Then  let $\gamma=\bigcup \{\rho_\alpha: \alpha<\omega_1\}$.  Then $\gamma$  is well-defined and any  $x\in 2^{\omega_1}$ which extends $\gamma$ is not in $\bigcup\{K_{(f, p)}: f\in \omega^\omega, p\in P\}$, which is a contradiction. \end{proof}

Next, using the result by Todor\v{c}evi\'{c} below, we obtain a more general result under the assumption $\mathfrak{d}>\mathfrak{b}=\aleph_1$.

\begin{lemma}\label{Tod}(Theorem 1.3, \cite{T89})
If $\mathfrak{b}=\aleph_1$, then $\omega^{\omega_1}$ has a subset  $X$ of cardinality $\aleph_1$, such that for any $\aleph_1$-sized subset $A$ of $X$, there exist a countable subset $D$ of $A$ and a $\gamma\in\omega_1$ such that $\pi_\gamma(D)$ is dense in $\omega^{\omega_1\setminus \gamma}$.
\end{lemma}

Note that the lemma above also holds with $\omega$ replaced by $2$; simply map $\omega^{\omega_1}$ onto $2^{\omega_1}$ by taking all coordinates modulo $2$.

Let $(M, d)$ be a separable metric space. Consider the Hausdorff metric $d^H$ defined on $\mathcal{K}(M)$: $$d^H(F, F')=\sup\{d(f, F'), d(F, f'): f\in F, f'\in F'\}.$$

It is well-known (see, e.g.,  \cite{C74}) that the space $(\mathcal{K}(M), d^H)$ is a separable metric space. First, we prove a result which will be used later.
\begin{lemma}\label{haus_metric}
Let $d$ be  a metric on a separable metric space $M$.

Then $d^H(F, \bigcup\{F_i: i=1,2, \ldots, n\})\leq \sup\{d^H(F, F_i): i=1, 2, \ldots, n\}$ given $F, F_1, \ldots, F_n\in \mathcal{K}(M)$.

\end{lemma}

\begin{proof} We will show that the result holds for $n=2$. Then inductively we obtain the general result.

Fix $F, F_1, F_2 \in \mathcal{K}(M)$. First, it is straightforward to verify that  $d(f, F_1\cup F_2)=\min\{d(f, F_i): i=1, 2\}$ for any $f\in F$. Then just note that $d^H(F, F_1\cup F_2)=\sup\{d(f, F_1\cup F_2), d(F, f'): f\in F, f'\in F_1\cup F_2\}$ and
$\sup\{d^H(F, F_i): i=1, 2\}=\sup\{d(f, F_1), d(f, F_2), d(F, f'): f\in F, f'\in F_1\cup F_2\}$. The result follows.
\end{proof}

\begin{theorem}\label{d>b=w1}
Assume $\mathfrak{d}>\mathfrak{b}=\aleph_1$. Let $\{K_F: F\in \mathcal{K}(M)\}$ be a $\mathcal{K}(M)$-directed compact cover of $2^{\omega_1}$ for some separable metric space $M$.

Then there exists $F\in \mathcal{K}(M)$ such that $K_F$ is BIG.

\end{theorem}

\begin{proof}

Applying Lemma~\ref{Tod}, we fix an $\aleph_1$-sized subset $X$ of $2^{\omega_1}$ such that for any $\aleph_1$-sized subset $A$ of $X$, there exist a countable subset $D$ of $A$ and a $\gamma\in\omega_1$ such that $\pi_\gamma(D)$ is dense in $2^{\omega_1\setminus \gamma}$. Then for any $x\in X$, we can find an $F_x\in\mathcal{K}(M) $ such that $x\in K_{F_x}$. Let $d^H$ be the Hausdorff metric defined on $\mathcal{K}(M)$. Then the set $\mathcal{X}=\{F_x: x\in X\}$ equipped with $d^H$ is a separable metric space. For each $x\in X$ and $n\in \mathbb{N}$, let $B(F_x, 1/n)$ denote the open ball containing $F_x$ with radius $1/n$.

We claim that there exists an $x\in X$ such that $K[B(F_{x}, 1/n)]$ contains uncountably many elements in $X$ for each $n\in \mathbb{N}$, where $K[B(F_{x_0}, 1/n)]=\bigcup\{K_F: F\in B(F_{x_0}, 1/n)\}$. Suppose not. For each $x\in X$, we fix $n_x$ such that $K[B(F_x, 1/n_x)]$ contains only countably many elements in $X$. Since $(\mathcal{X}, d^H)$ is a separable metric space, there exists a countable subcollection $\mathcal{B}=\{B(F_{x_i}, 1/n_{x_i}): i\in \omega\}$ of $\{B(F_x, 1/n_x): x\in X\}$ which covers $\mathcal{X}$. Then $\bigcup\{K[B(F_{x_i}, 1/n_{x_i})]: i\in \omega\}$ contains only countably many elements of $X$ by the property of $n_{x_i}$'s, but it also contains all elements of $X$ since $\mathcal{X} \subseteq \bigcup \{B(F_{x_i}, 1/n_{x_i}): i\in \omega\}$. We have reached a contradiction.

Now fix an $x_0\in X$ such that for each $n\in \mathbb{N}$, $K[B(F_{x_0}, 1/n)]$ contains uncountably many elements in $X$. For each $n\in \mathbb{N}$,  define $A_n=K[B(F_{x_0}, 1/n)]\cap X$. By Lemma~\ref{Tod}, we can find a countable subset $D_n$ of $A_n$ and a $\delta_n\in\omega_1$ such that $\pi_\delta(D_n)$ is dense in $2^{\omega_1\setminus \delta_n}$. List $D_n$ as $\{d(m, n): m\in \omega\}$ for each $n\in \mathbb{N}$.

Define $\delta=\sup\{\delta_n: n\in \mathbb{N}\}$. For any $s\in \text{Fn}(\omega_1\setminus \delta, 2)$, define $h_s(n)=\min\{m: d(m, n)\in [s]\}$. Then we obtain an uncountable subset $\{h_s: s\in \text{Fn}(\omega_1\setminus \delta, 2)\}$ of $\omega^\omega$. Since $\mathfrak{d}>\mathfrak{b}=\aleph_1$, there exists a function $g\in \omega^\omega$ such that for each $s\in \text{Fn}(\omega_1\setminus \delta, 2)$, $g(i)> h_s(i)$ for infinitely many $i\in \mathbb{N}$. Let $E=\{d(m, n): m\leq g(n)\}$. We see that for any $s\in \text{Fn}(\omega_1\setminus \delta, 2)$, $[s]\cap E\neq \emptyset$, hence $\pi_\delta(E)$ is dense in $2^{\omega_1\setminus \delta}$.

For each $d(m, n)$ with $m\leq g(n)$, there exists an $F_{m, n}\in B(F_{x_0}, 1/n)$ such that $d(m, n)\in K_{F_{m,n}}$. Then by Lemma~\ref{haus_metric}, $d^H(F_{x_0}, \bigcup\{F_{m, n}: m\leq g(n)\})<1/n$. Hence, $\hat{F}=\bigcup\{F_{m, n}: m<g(n), n\in \mathbb{N}\}\cup F_{x_0}$ is a compact subset of $M$. Also  $E\subset K_{\hat{F}}$. Hence $K_{\hat{F}}$ is BIG witnessed by $\delta$.\end{proof}

Combining Theorem~\ref{d=w1orb>w1} and Theorem~\ref{d>b=w1}, we obtain our main result in this section.

\begin{theorem}\label{k(q)_d_b} If $2^{\omega_1}$ has a $\mathcal{K}(\mathbb{Q})$-directed compact cover, then there exists an  $F\in \mathcal{K}(\mathbb{Q})$ such that $K_{F}$ is BIG. \end{theorem}

\begin{proof} Let $\{K_F: F\in \mathcal{K}(\mathbb{Q})\}$ be a $\mathcal{K}(\mathbb{Q})$-directed compact cover of $2^{\omega_1}$.

If $\mathfrak{d}>\mathfrak{b}=\aleph_1$, then there exists an  $F\in \mathcal{K}(\mathbb{Q})$ such that $K_{F}$ is BIG by Theorem~\ref{d>b=w1}.

Since $\mathcal{K}(\mathbb{Q})$ is a Tukey quotient of  $\omega^\omega\times [\omega_1]^{<\omega}$, we fix a order-preserving mapping $\phi: \omega^\omega\times [\omega_1]^{<\omega}\rightarrow \mathcal{K}(\mathbb{Q})$ such that $\phi(\omega^\omega\times [\omega_1]^{<\omega})$ is cofinal in  $\mathcal{K}(\mathbb{Q})$. Then for each $p\in \omega^\omega\times [\omega_1]^{<\omega}$, we define $K_p=K_{\phi(p)}$. Then we get a $\omega^\omega\times [\omega_1]^{<\omega}$-directed compact cover of $2^{\omega_1}$. Hence by Theorem~\ref{d=w1orb>w1}, there is a $p\in \omega^\omega\times [\omega_1]^{<\omega}$ with $K_p$ (hence, $K_{\phi(p)}$)  being BIG under the assumption
that $\mathfrak{d}=\aleph_1$ or $\mathfrak{b}>\aleph_1$. \end{proof}

\section{Compact Spaces with a $\K(\Q)$-diagonal}

We now give a sufficient condition for  $X^2\setminus \Delta$, where $X$ is compact, to not have a $P$-directed compact cover for some directed set $P$ with calibre~($\omega_1, \omega$). Then we deduce our main result.

\begin{theorem}\label{b_nPd} Let $P$ be a directed set with calibre~$(\omega_1, \omega)$ and $X$ a compact space that maps continuously onto $2^{\omega_1}$.

Suppose that for any $P$-directed compact cover $\{K_p:p\in P\}$ of $2^{\omega_1}$, there is a $p\in P$ such that $K_p$ is BIG. Then $X^2\setminus \Delta$ does not have a $P$-directed compact cover.

\end{theorem}

\begin{proof} For any node $\rho\in 2^{<\omega_1}$, it is clear that $[\rho]$ is homeomorphic to $2^{\omega_1}$. Hence for any $P$-directed compact cover $\{K_p:p\in P\}$ of $[\rho]$, there is a $p\in P$ such that $K_p$ is BIG. Therefore by Lemma~\ref{b-base}, for any BIG subset $Z$ of $2^{\omega_1}$, if $\{K_p:p\in P\}$ is a $P$-directed compact cover of $Z$, then there is a $p\in P$ such that $K_p$ is BIG.

Let $\phi$ be a continuous map from $X$ onto $2^{\omega_1}$. For any subset $B$ of $X$, we say that $B$ is BIG if it is closed and $\phi(B)$ is a BIG subset of $2^{\omega_1}$. It is straightforward to verify that if $B$ is a BIG subset of $X$ and $y\in B$, then $B\setminus \{y\}$ still contains a BIG subset of $X$.

\medskip

\textbf{Claim ($\ast$):} If $\{B_n: n\in\omega\}$ is a decreasing sequence of BIG subsets of $X$, then $\bigcap\{B_n:n\in \omega\}$ is also a BIG subset of $X$.

\textbf{Proof of Claim ($\ast$):} It is clear that $\bigcap\{B_n:n\in \omega\}$ is closed and $\bigcap\{\phi(B_n):n\in \omega\}\supset \phi(\bigcap\{B_n: n\in \omega\})$. Since $\bigcap\{\phi(B_n):n\in \omega\}$ is a BIG subset of $2^{\omega_1}$ by Lemma~\ref{ctbl_insctn_big}, it is enough to show that $\bigcap\{\phi(B_n):n\in \omega\}\subset \phi(\bigcap\{B_n: n\in \omega\})$. Take any $x\in \bigcap\{\phi(B_n):n\in \omega\}$. For each $n$, there exists $y_n\in B_n$ such that $\phi(y_n)=x$. Since $X$ is compact, $\{y_n: n\in\omega\}$ has a cluster point $y$. We claim that $y\in \bigcap\{B_n: n\in \omega\}$, i.e. $x\in \phi(\bigcap\{B_n: n\in \omega\})$. Suppose not. Then there exists an $n_0\in \omega$ such that $y\notin B_{n_0}$. Let $U=X\setminus B_{n_0}$. Then $U$ is an open neighborhood of $y$ such that $y_n\notin U$ for all $n\geq n_0$. This contradicts with the fact that $y$ is a cluster point of $\{y_n: n\in \omega\}$. This finishes the proof of the claim.

\medskip

We will prove the statement of the theorem next. Suppose, for a contradiction, that $X$ has a $P$-diagonal witnessed by $\mathcal{C}=\{K_p:p\in P\}$.
Let $Y_0=X$ and choose $y_0\in Y_0$ and $p_0\in P$.  Let $i_0\in 2$ be such that
$\phi(y_0)(0)\neq i_0$.  Then $\phi^{-1}([<0, i_0>])$ is BIG and does not contain $y_0$, so $\{y_0\}\times \phi^{-1}([<0, i_0>])\subset X^2\setminus \Delta$. (Recall that $[<0, i_0>]=\{x\in 2^{\omega_1}: <0, i_0>\in x\}=\{x\in 2^{\omega_1}: x(0)=i_0\}$.)  Then $\{K_p\cap (\{y_0\}\times \phi^{-1}([<0, i_0>])): p\in P\}$ is a $P$-directed compact cover of $\{y_0\}\times \phi^{-1}([<0, i_0>])$. Hence $\{\phi(\pi_2(K_p\cap (\{y_0\}\times \phi^{-1}([<0, i_0>])))): p\in P\}$ is a $P$-directed compact cover of $[<0,i_0>]$.  By assumption, there is a $p_1\in P$ such that $\phi(\pi_2(K_{p_1}\cap (\{y_0\}\times \phi^{-1}([<0, i_0>]))))$ is a BIG subset of $2^{\omega_1}$. Let $Y_1=\pi_2(K_{p_1}\cap (\{y_0\}\times \phi^{-1}([<0, i_0>]))$ and pick $y_1\in Y_1$. It is clear that $Y_1$ is BIG and $Y_1 \subseteq Y_0$. Also note that $\{y_0\}\times Y_1\subset K_{p_1}$.

%Then we pick $y_1\in Y_1$ and $i_1\in \{0, 1\}$ with $i_1\neq \phi{y_1}$. Let $Y_2=\Y_1\cap \phi^{\leftarrow}[(\delta_1, i_i) ]$, hence $\{y_1\}$

Fix $\alpha\in \omega_1$ and assume that for all  $\beta< \alpha $, $y_\beta$, $Y_\beta$,  and $p_\beta$ are defined as follows:
\begin{itemize}
\item[1)] $y_\beta\in Y_\beta$;
\item[2)] $Y_\beta\subseteq Y_\gamma$ for all $\gamma \leq \beta$;
\item[3)] $\{y_\gamma\}\times Y_{\gamma+1}\subset K_{p_{\gamma+1}}$ for all $\gamma <\beta$;
\item[4)] $Y_\beta$ is BIG.
\end{itemize}

If $\alpha$ is a limit ordinal, let $Y_\alpha=\bigcap\{Y_\beta: \beta<\alpha\}$.  Then $Y_\alpha$ is the intersection of a countable decreasing collection of BIG subsets of $X$, so it is BIG by Claim ($\ast$).
Choose $y_\alpha\in Y_\alpha$ and $p_\alpha\in P$.

  Now assume $\alpha$ is a successor, say $\alpha=\alpha_- +1.  $ By Lemma~\ref{b-base}, we fix $\rho_{\alpha}\in 2^{<\omega_1}$ such that $[\rho_{\alpha}]\subseteq \phi(Y_{\alpha_-})$. Take $\eta_\alpha\in \omega_1\setminus \text{Dom}(\rho_{\alpha})$ and $i_\alpha\in \{0, 1\}$ such that $\phi(y_{\alpha_-})(\eta_\alpha)\neq i_\alpha$.  The set $Y_{\alpha_-}\cap \phi^{-1}([\rho_\alpha\cup\{<\eta_\alpha, i_\alpha>\}])$  is clearly BIG and doesn't contain $y_{\alpha_-}$. Hence $\{K_p\cap (\{y_{\alpha_-}\}\times (Y_{\alpha_-}\cap \phi^{-1}([\rho_\alpha\cup\{<\eta_\alpha, i_\alpha>\}])): p\in P\}$ is a $P$-directed compact cover of $\{y_{\alpha_-}\}\times (Y_{\alpha_-}\cap \phi^{-1}([\rho_\alpha\cup\{<\eta_\alpha, i_\alpha>\}]))$. Therefore $\{\phi(\pi_2(K_p\cap (\{y_{\alpha_-}\}\times (Y_{\alpha_-}\cap \phi^{-1}([\rho_\alpha\cup\{<\eta_\alpha, i_\alpha>\}])))): p\in P\}$ is a $P$-directed compact cover of $[\rho_\alpha\cup\{<\eta_\alpha, i_\alpha>\}]$.   Then by assumption, there is a $p_\alpha\in P$ such that $\phi(\pi_2(K_{p_\alpha}\cap (\{y_{\alpha_-}\}\times (Y_{\alpha_-}\cap \phi^{-1}([\rho_\alpha\cup\{<\eta_\alpha, i_\alpha>\}]))))$ is BIG in $2^{\omega_1}$. Let $Y_\alpha=\pi_2(K_{p_\alpha}\cap (\{y_{\alpha_-}\}\times (Y_{\alpha_-}\cap \phi^{-1}([\rho_\alpha\cup\{<\eta_\alpha, i_\alpha>\}])))$ which is clearly BIG, so condition 4) holds. It is straightforward to verify that $Y_{\alpha_-}\supseteq Y_\alpha$ and $\{y_{\alpha_-}\}\times Y_\alpha \subset K_{p_\alpha}$, so conditions 2) and 3) hold. Take $y_\alpha\in Y_\alpha$.

%$Y_\alpha$ of $\bigcap\{Y_\beta: \beta<\alpha\}$ such that $y_\alpha\notin Y_\alpha$. Then  we obtain a BIG subset $Y_{\alpha+1}$ of $Y_{\alpha}$ and $p_{\alpha+1}\in P$ such that $\{y_\alpha\}\times Y_{\alpha+1}\subset K_{p_{\alpha+1}}$. Then we pick $y_{\alpha+1}\in Y_{\alpha+1}$.

%we pick $y_{\alpha}\in \bigcap\{Y_\beta: \beta<\alpha\}$ and $i_\alpha\in \{0, 1\}$ with $i_\alpha\neq \phi_{\gamma}(y_\alpha)$. We can see that $ (\bigcap\{Y_\beta: \beta<\alpha\}\cap \phi^{-1}[<\gamma, i_\alpha>]$ is BIG and $y_\alpha\notin (\bigcap\{Y_\beta: \beta<\alpha\}\cap \phi^{-1}[<\gamma, i_\alpha>]$. Hence  the set $\{y_\alpha\}\times (\bigcap\{Y_\beta: \beta<\alpha\}\cap \phi^{-1}[<\gamma, i_\alpha>]$ has a $P$-ordered compact cover naturally from $\mathcal{C}$. By Lemma~\ref{ctbl_insctn_big}, using the previous argument,  there exists a BIG set $Y_\alpha\subseteq (\bigcap\{Y_\beta: \beta<\alpha\}) \cap \phi^{-1}[<\gamma, i_\alpha>]$ and $p_\alpha\in P$ such that $\{y_\alpha\}\times Y_\alpha \subset K_{p_\alpha}$. Let $\gamma_\alpha$ be the ordinal which witnesses the BIGness of $Y_\alpha$.

Therefore by transfinite induction, we obtain an uncountable collection $\{y_\alpha, Y_\alpha, p_\alpha: \alpha<\omega_1\}$ such that conditions 1) - 4) are satisfied. Since $P$ has calibre $(\omega_1, \omega)$, there exists a countably infinite subcollection of $\{p_{\alpha+1}: \alpha\in \omega_1\}$ which is bounded.  Without loss of generality, we can list this subcollection as $\{p_{\alpha_n+1}:n\in\omega\}$ such that $\alpha_n < \alpha_{n+1}$ for all $n\in \omega$. Fix $q\in P$ such that $p_{\alpha_n+1}\leq q$ for all $n<\omega$; then $K_{p_{\alpha_n+1}}\subset K_q$ for all $n\in \omega$. Suppose $n<m$.  Then by conditions 2) and 3), $(y_{\alpha_n}, y_{\alpha_m})\in \{y_{\alpha_n}\}\times Y_{\alpha_m}\subset \{y_{\alpha_n}\}\times Y_{\alpha_n+1}\subset K_{p_{\alpha_n+1}}\subset K_q$. So we see that $(y_{\alpha_n}, y_{\alpha_m})\in  K_q$ for any $m, n\in \omega$ with $n<m$. Let $y$ be a cluster point of $\{y_{\alpha_n}: n\in \omega\}$.  Then by the compactness of $K_q$, $(y_{\alpha_n}, y)\in K_q$ for all $n\in\omega$. Therefore $(y, y)\in K_q$. However, $K_q$ is disjoint from the diagonal of $X$ by assumption, so we have reached a contradiction. \end{proof}

%\begin{lemma} Let $\{Y_\alpha:\alpha\in\omega_1\}$ be a collection of Lindel\"{o}f subspaces of $2^{\omega_1}$. If $2^{\omega_1}=\bigcup\{Y_\alpha: \alpha\in \omega_1\}$, then there is an $\alpha\in \omega_1$ such that $Y_\alpha$ is FAT.\end{lemma}
%\begin{proof} Suppose, for contradiction, that $Y_\alpha$ is not FAT for each $\alpha<\omega_1$. Then for each node $\delta$ in $2^{<\omega_1}$ and $\alpha$ in $\omega_1$ , we have that $[\delt\begin{lemma} If a compact subset $A$ of $2^{\omega_1}$ is not BIG, then it is THIN.\end{lemma}a]\setminus Y_\alpha\neq \emptyset$.\end{proof}

Since any directed set with calibre~$\omega_1$ also has calibre~$(\omega_1,\omega)$, the result above also holds for any directed set with calibre~$\omega_1$.

Since $\mathcal{K}(M)$ has calibre~$(\omega_1, \omega)$ for any separable metric space $M$, we obtain the following corollary.

\begin{corollary}\label{nobig} Let $X$ be a compact space that maps onto $2^{\omega_1}$.

Let $M$ be a separable metric space such that for any $\mathcal{K}(M)$-directed compact cover $\{K_F: F\in \mathcal{K}(M)\}$ of $2^{\omega_1}$ , there is an $F\in \mathcal{K}(M)$ such that $K_F$ is BIG. Then $X$ does not have an $M$-diagonal.

\end{corollary}

\begin{theorem}\label{main} Any compact space with a  $\mathbb{Q}$-diagonal is metrizable.
\end{theorem}

\begin{proof} Let $X$ be a compact space with a $\mathbb{Q}$-diagonal.

If the tightness of $X$ is countable, $X$ is metrizable by Theorem~2.9 in \cite{COT11}. Assume that $X$ has uncountable tightness. Then there exists a closed subset $Y$ of $X$ that maps continuously onto $\omega_1+1$. We fix a continuous onto map $\psi: Y\rightarrow \omega_1+1$ and define $Y'=\{x\in Y: \psi(x)<\omega_1\}$. Then $Y'$ is $\omega$-bounded. Let $\mathcal{C}=\{K_F: F\in\mathcal{K}(\mathbb{Q})\}$ be a $\mathcal{K}(\mathbb{Q})$-directed compact cover of $X^2 \setminus \Delta$. Then $\mathcal{D}=\{K_F\cap (Y'^2 ): F\in \mathcal{K}(\mathbb{Q})\}$ is a $\mathcal{K}(\mathbb{Q})$-directed closed cover of $Y'^2\setminus \Delta$. By
Theorem 3.1 in \cite{ST16}, there is a compact separable subspace $Z$ of $Y'$ which maps continuously onto $\mathbb{I}^{\omega_1}$. Hence we get a compact subspace $Z'$ of $Z$ which maps continuously onto $2^{\omega_1}$. We see that $Z'$ inherits a $\mathbb{Q}$-diagonal from $X$. However, by Theorem~\ref{k(q)_d_b} and Corollary~\ref{nobig}, $Z'$ does not have a $\mathbb{Q}$-diagonal which is a contradiction.  \end{proof}

\begin{corollary} Let $P$ be a directed set with $\mathcal{K}(\mathbb{Q})\geq_T P$. If a space $X$ is compact and $X^2\setminus \Delta$ has a $P$-directed compact cover, then $X$ is metrizable.
\end{corollary}

The proof of the following result is almost identical with the proof of Theorem~\ref{main}.
\begin{theorem}\label{smspace} Assume $\mathfrak{d}>\mathfrak{b}=\aleph_1$. Any compact space with an $M$-diagonal for some separable metric space $M$ is metrizable.
\end{theorem}

The theorem below provides a positive answer to Problem 4.8 of the paper \cite{COT11}. Its proof is analogous to the one of Theorem 3.4 in \cite{ST16} but we give it here anyway for the sake of completeness.

A family $\mathcal{N}$ is a network with respect to a collection $\mathcal{C}$ of subsets of $X$ if for any $C\in\mathcal{C}$ and open set $U$ in $X$ containing $C$, there exists an $N\in\mathcal{N}$ such that $C\subset N\subset U$. A space $X$ is a Lindel\"of $\Sigma$-space if there exists a countable family $\mathcal{F}$ of subsets of $X$ such that $\mathcal{F}$ is a network with respect to a compact cover $\mathcal{C}$ of the space $X$.

\begin{theorem}\label{main2} Let $X$ be a Tychonoff space with an $M$-diagonal for some separable metric space $M$. Then

\begin{itemize}
\item[1)] if $\mathfrak{d}>\mathfrak{b}=\aleph_1$, then $X$ is cosmic.
\item[2)] if $M$ is the space $\mathbb{Q}$, then $X$ is cosmic.

\end{itemize}
\end{theorem}

\begin{proof}We see that every compact subspace of $X$ is metrizable using Theorem~\ref{main} and Theorem~\ref{smspace}.

Let $\{K_F: F\in \K(M)\}$ be a $\mathcal{K}(M)$-directed compact cover of $X^2\setminus \Delta$. By Proposition 2.6 in \cite{COT11}, there exists a collection $\mathcal{C}=\{C_F: F\in \K(M)\}$ of subsets of $X^2\setminus \Delta$ such that each $C_F$ is $\omega$-bounded and there is a countable network $\mathcal{N}$ with respect to $\mathcal{C}$. More specifically, $C_F$ is countably compact for each $F\in \K(M)$.

We claim that $C_F$ is metrizable for each $F\in \K(M)$. Suppose not. Pick an $F_0\in \K(M)$ such that $C_{F_0}$ is non-metrizable. Since the projections of $C_{F_0}$ on $X$ are $\omega$-bounded, we can choose $D\subset X$ such that $D$ is $\omega$-bounded and $C_{F_0}\subset D\times D$. Clearly, $D\times D$ is non-metrizable, hence $D$ is not compact. Since $D^2\setminus \Delta$ has a $\K(M)$-directed cover of closed subsets, applying Theorem~3.1 in \cite{ST16}, we can find a compact subset $Z$ of $D$ which maps continuously onto $2^{\omega_1}$. Then $Z$ has an $M$-diagonal inherited from $X$, but by Theorem~\ref{b_nPd}, Theorem~\ref{k(q)_d_b} and Theorem~\ref{d>b=w1}, $Z$ doesn't have an $M$-diagonal. This is a contradiction.

For each $F\in \K(M)$, $C_F$ is compact since it is countably compact and metrizable. Hence $X^2\setminus \Delta$ is a Lindel\"of $\Sigma$-space witnessed by $\{C_F: F\in \K(M)\}$ and $\mathcal{N}$. Then $X$ is also a Lindel\"of $\Sigma$-space. For each $C_F\in \mathcal{C}$, there exists an $N\in \mathcal{N}$ such that $C_F\subset N \subset \overline{N}\subset X^2\setminus \Delta$. Since $\mathcal{N}$ is countable, $X$ has a $G_\delta$-diagonal. Hence $X$ is cosmic by Theorem~2.1.8 in \cite{Ar78} and Problem 266 in \cite{TK14}. \end{proof}

Since any pseudocompact space with a countable network is compact, we obtain the following corollary.

\begin{corollary} Let $X$ be a pseudocompact space with an $M$-diagonal for some separable metric space $M$. Then

\begin{itemize}
\item[1)] if $\mathfrak{d}>\mathfrak{b}=\aleph_1$, then $X$ is metrizable.
\item[2)] if $M$ is the space $\mathbb{Q}$, then $X$ is metrizable.

\end{itemize}
\end{corollary}

\begin{example} There is a non-metrizable compact space $X$ such that  $X^2\setminus \Delta$ has  a $\omega^\omega\times [\omega_1]^{<\omega}$-directed compact cover.
\end{example}

\begin{proof}Let $A(\omega_1)$ be the one-point compactification of $\aleph_1$-many discrete points. Fix a one-to-one mapping $\phi: \omega_1\rightarrow A(\omega_1)^2\setminus \Delta$. Define $\mathcal{C}=\{\phi(F): F\in [\omega_1]^{<\omega}\}$ which is an $[\omega_1]^{<\omega}$-directed compact cover of $A(\omega_1)^2\setminus \Delta$.  Then since $[\omega_1]^{<\omega}\leq_T \omega^\omega\times [\omega_1]^{<\omega}$, we see that $A(\omega_1)^2\setminus \Delta$ has a $ \omega^\omega\times [\omega_1]^{<\omega}$-directed compact cover. \end{proof}

\subsection*{Acknowledgement}

The author would like to thank Gary Gruenhage and the referee for the valuable comments and suggestions which improved the exposition.

%%%%%%%%%%% To ease editing, use normal size for the references:

\end{document}